\documentclass[12pt]{article}

\usepackage{amsmath,amsfonts,amssymb,amsthm,bm}
\usepackage[a4paper]{geometry}
\usepackage{times}

\newtheorem{Def}{Definition}
\newtheorem{thm}{Theorem}
\newtheorem{lem}[thm]{Lemma}

\newtheorem{cor}[thm]{Corollary}
\newtheorem{Rem}{Remark}

\newcommand{\cH}{{\cal H}}

\newcommand{\cQ}{{\cal Q}}

\newcommand{\ZZ}{\mathbb{Z}}
\newcommand{\RR}{\mathbb{R}}
\newcommand{\NN}{\mathbb{N}}
\newcommand{\CC}{\mathbb{C}}

\newcommand{\Ab}{{\boldsymbol{A}}}
\newcommand{\Bb}{{\boldsymbol{B}}}
\newcommand{\Cb}{{\boldsymbol{C}}}
\newcommand{\Db}{{\boldsymbol{D}}}

\newcommand{\Fb}{{\boldsymbol{F}}}
\newcommand{\Gb}{{\boldsymbol{G}}}
\newcommand{\Hb}{{\boldsymbol{H}}}
\newcommand{\Ib}{{\boldsymbol{I}}}

\newcommand{\Mb}{{\boldsymbol{M}}}
\newcommand{\Nb}{{\boldsymbol{N}}}

\newcommand{\Qb}{{\boldsymbol{Q}}}
\newcommand{\Rb}{{\boldsymbol{R}}}

\newcommand{\Tb}{{\boldsymbol{T}}}

\newcommand{\Vb}{{\boldsymbol{V}}}
\newcommand{\Wb}{{\boldsymbol{W}}}
\newcommand{\Xb}{{\boldsymbol{X}}}
\newcommand{\Yb}{{\boldsymbol{Y}}}
\newcommand{\Zb}{{\boldsymbol{Z}}}

\newcommand{\cb}{{\boldsymbol{c}}}
\newcommand{\db}{{\boldsymbol{d}}}
\newcommand{\eb}{{\boldsymbol{e}}}

\newcommand{\gb}{{\boldsymbol{g}}}

\newcommand{\vb}{{\boldsymbol{v}}}

\newcommand{\bPhi}{{\boldsymbol{\Phi}}}

\newcommand\bZero{{\boldsymbol 0}}

\newcommand{\diag}{{\mathop{\mbox{\rm diag}\,}}}
\newcommand{\supp}{{\mathop{\mbox{\rm supp}\,}}}

\begin{document}

\title{Level-dependent interpolatory Hermite subdivision schemes and wavelets}

\author{
  {Mariantonia Cotronei}\thanks{DIIES, Universit\`a Mediterranea di
    Reggio Calabria, Via Graziella loc. Feo di Vito, 89122 Reggio
    Calabria, Italy. \texttt{mariantonia.cotronei@unirc.it}}  
\and 
{Caroline Moosm\"uller}\thanks{Department of Chemical and Biomolecular
  Engineering, Johns Hopkins University, 3400 North Charles Street,
  Baltimore, MD 21218, USA. \texttt{cmoosmueller@jhu.edu}} 
\and 
{Tomas Sauer}\thanks{Lehrstuhl f\"ur Mathematik mit Schwerpunkt Digitale
  Signalverarbeitung \& FORWISS, Universit\"at Passau,
  Fraunhofer IIS Research Group on Knowledge Based Image
  Processing, Innstr.~43,
  94032 Passau, Germany. \texttt{tomas.sauer@uni-passau.de}}
\and
{Nada Sissouno}\thanks{Department of Mathematics, Technical University
  of Munich, Boltzmannstra\ss e 3, 85748 Garching,
  Germany. \texttt{sissouno@ma.tum.de}}
}

\maketitle
\begin{abstract}
  We study many properties of level-dependent Hermite
  subdivision, focusing on schemes preserving polynomial and exponential
  data. We specifically consider interpolatory schemes, which give rise
  to level-dependent multiresolution analyses through a
  prediction-correction approach. A result on the decay of the
  associated multiwavelet coefficients, corresponding to an uniformly
  continuous and differentiable function, is derived. It makes use of
  the approximation of any such function with a generalized Taylor
  formula expressed in terms of polynomials and exponentials. 
  
\end{abstract}

\section{Introduction}
Hermite subdivision schemes are iterative procedures that allow,
through a refinement process, to generate curves 
from a given set of discrete data points, consisting of function and derivative
values. Such data naturally occurs in applications of motion control
where position, velocity, acceleration and even higher derivatives of
the motion are computed on a discrete grid,
cf. \cite{hoffmann06:_method_and_devic_for_guidin}, and have then to
be interpolated by the Numerical Control when the motion is actually
performed.

The polynomial reproduction property of such schemes has been
thoroughly investigated in the last years, cf.
\cite{merrien14,dubuc06,dubuc05,dyn95,han05}. As in standard (non-Hermite)
schemes, such a property is crucial for assuring  not only the
convergence of the scheme, but also the smoothness and the
approximation order of its limit function \cite{conti16b,dyn90,dyn02}.  

Recently, some research has focused both on standard
\cite{conti11a,dyn03} and Hermite \cite{conti16,conti17a,conti15,uhlmann}
schemes preserving not only polynomial, but also exponential data,
that is, sequences of the form $\left( e^{\lambda k} : k \in \ZZ
\right)$. This generalization allows the generation of curves which
also exhibit transcendental features. Such schemes necessarily have a
level-dependent nature, which means that the subdivision operator
varies at each step of the subdivision process.  

The reproduction property of level-dependent schemes has been explored
from the point of view of wavelet analysis \cite{dyn14,vonesch07}, where it translates into a 
vanishing moment property for both exponentials and polynomials.
Using the strong connection between wavelets and subdivision schemes,
\cite{cotronei17} proposes a construction of Hermite multiwavelets and
corresponding multiresolution analysis (MRA) with polynomial and
exponential vanishing moments. This construction is based on the
interpolatory Hermite schemes possessing the polynomial and
exponential preservation property introduced in \cite{conti16}. 

We remark that the vanishing moment property is very desirable in
wavelet analysis, as it ensures compression capabilities of the
wavelet system as long as the processed data contains many ``small''
negligible details and is of a certain smoothness otherwise.
How small such details are depends on the decay of the wavelet
coefficients and indeed, the decay rate can serve as a measure of
smoothness of the underlying function.
 
The aim of this paper is threefold. We start by providing some basic
results on level-dependent Hermite schemes 
in Sections~\ref{sec:preliminaries} and \ref{sec:MRA_Hermite}.
The main focus is on the properties of the limit functions of such schemes, which are the building blocks of
the corresponding level-dependent MRA. 

If such schemes possess the property of reproducing polynomial and
exponential data, then the wavelet coefficients satisfy a generalized
vanishing moment condition which is important for assuring sparse
representations of any
function $f\in C^d(\RR)$. This property is due to a certain decay rate
of the wavelet coefficients as the scale increases.
In order to prove such a decay, usually a Taylor expansion is used. We
thus introduce and analyze in the second part of this paper, Section~\ref{sec:taylor},
a generalized Taylor formula, which expands a given function using
polynomials and exponentials. We also compare this generalization to
the classical Taylor formula and derive an error bound between the two.

With this result at hand, we are able to
determine the decay of the wavelet
coefficients connected to the MRA generated by level-dependent
\emph{interpolatory} Hermite subdivision procedure. 
This constitutes the third and final part of this paper, namely
Section~\ref{sec:decay}.

Our analysis is tailored for \emph{interpolatory} Hermite schemes from
which it is very easy and natural to construct wavelet systems that
depend only on point evaluations by
means of prediction-correction methods. Moreover, from the viewpoint
of extending our results also to manifold-valued data, this is a
reasonable approach:
In \cite{grohs12} it is shown that in general manifold-valued
analogues of even scalar, non-Hermite wavelets are not possible any more, while
perfect reconstruction, stability and wavelet coefficient decay
results can still be achieved for many interpolatory examples, see
\cite{grohs09,grohs12}. We do expect that our result on the decay of
Hermite multiwavelet coefficients (Theorem~\ref{T:leveldep_decay}) 
can be transferred to the manifold setting by combining it with
results from \cite{grohs09} and \cite{moosmueller16,moosmueller17}.
 
\section{Preliminaries and first results}\label{sec:preliminaries}
We start by fixing the notation and introducing some basics concepts.
Vectors and matrices in $\RR^{d+1}$ and in $\RR^{(d+1) \times (d+1)}$
are denoted by boldface lower and upper case letters, respectively. If
the particular coordinates are of interest, a column vector $\vb$ is
also written as $\vb=\left[v_j \right]_{j=0}^d$. 
For the canonical basis in $\RR^{d+1}$ we write $\eb_j$,
$j=0,\ldots,d$. 
On $\RR^{d+1}$ we use both the infinity-norm $| \cdot |_{\infty}$ and
the Euclidean norm $| \cdot |_{2}$, while for matrices in $\RR^{(d+1)
  \times (d+1)}$ we use the operator norms, induced by the respective
norms on $\RR^{d+1}$ and again denote them by $| \cdot |_{\infty}$ and
$| \cdot |_{2}$. 

The space of all vector-valued sequences $\cb=(\cb_j: j\in \ZZ)$ is
written as $\ell(\ZZ)^{d+1}$,
and $\ell (\ZZ)^{(d+1) \times (d+1)}$ stands for the space of all
matrix-valued sequences $\Ab=(\Ab_j: j\in \ZZ)$.
By $\ell_{\infty}(\ZZ)^{d+1}$ we denote all bounded vector-valued
sequences, i.e., sequences $\cb \in \ell(\ZZ)^{d+1}$ with
\begin{equation*}
\|\cb\|_{\infty} := \sup_{j \in \ZZ}|\cb_j|_{\infty} < \infty.
\end{equation*}
Similarly, $\ell_{\infty}(\ZZ)^{(d+1) \times (d+1)}$
is the space of matrix-valued sequences where
\begin{equation*}
\|\Ab \|_{\infty} := \sup_{j \in \ZZ}|\Ab_j|_{\infty} < \infty.
\end{equation*}
The \emph{symbol} of a finitely supported matrix
sequence, which we write as $\Ab \in \ell_0 (\ZZ)^{(d+1) \times (d+1)}$, is the
matrix valued \emph{Laurent polynomial}
\[
\Ab^* (z) = \sum_{j \in \ZZ} \Ab_j \, z^j, \qquad z \in \CC \setminus
\{0\}.
\]
The scalar peak sequence $\delta \in \ell_0 (\ZZ)$ with $\delta_j =
\delta_{j0}$, $j \in \ZZ$, can also be used to build matrix valued
sequences $\delta \, \Cb$ of any given dimension, where $( \delta \, \Cb
)_0 = \Cb$ and all other values equal to the zero matrix.

The space of all $d$-times continuously differentiable functions $f: \RR \to \RR$ is denoted by $C^d(\RR)$. Similarly, $C_u(\RR)$ is the space of uniformly continuous and bounded functions, whereas $C^d_u(\RR)$
contains all $d$-times continuously differentiable functions with derivatives $f^{(j)} \in C_u(\RR)$, $j=0,\ldots,d$. On $C_u(\RR)$ we use the infinity-norm
\begin{equation*}
\|f\|_{\infty}=\sup_{x \in \RR}|f(x)|,
\end{equation*}
while for a vector-valued function $\gb \in C_u(\RR)^{d+1}$ we employ the norm
\begin{equation*}
\|\gb\|_{\infty}=\sup_{x \in \RR}|\gb(x)|_{\infty}.
\end{equation*}
In the case that $\gb=\left[ g^{(j)} \right]_{j=0}^d$, $g \in
C_u^{d}(\RR)$, we so obtain the Sobolev norm
\begin{equation*}
\|\gb\|_{\infty}=\|g\|_{d,\infty}=\max_{j=0,\ldots,d}\|g^{(j)}\|_{\infty}.
\end{equation*}
For matrix-valued functions $\Gb \in C_u(\RR)^{(d+1) \times (d+1)}$,
the norm is given by the matrix infinity-norm 
\begin{equation*}
\|\Gb\|_{\infty}=\sup_{x\in \RR}|\Gb(x)|_{\infty}.
\end{equation*}

\subsection{Hermite subdivision schemes}
For $n \in \NN$, we define the level-$n$ \emph{subdivision operator}
$S_{\Ab^{[n]}}: \ell(\ZZ)^{d+1} \to \ell(\ZZ)^{d+1}$ with finitely
supported matrix mask $\Ab^{[n]} \in \ell_0 (\ZZ)^{(d+1) \times (d+1)}$
as 
\begin{equation}\label{eq:subdivop}
 (S_{\Ab^{[n]}}\cb)_j=\sum_{k \in \ZZ} \Ab^{[n]}_{j-2k}\cb_k, \qquad j \in \ZZ,\, \cb \in  \ell(\ZZ)^{d+1}.
\end{equation}
One immediately notices the well-known fact that the subdivision
operator is a composition of the  upsampling operator $\uparrow$,
defined as $(\uparrow \Bb)_{2j} = \Bb_{j}$ and
$(\uparrow \Bb)_{2j+1} = \bZero$, $j \in \ZZ$, 
$\Bb \in \ell(\ZZ)^{(d+1) \times (d+1)} $ and
the convolution $\ast$. In view of this, we denote by $*_2$ the
following operation
between two finitely supported masks $\Ab,\Bb \in \ell_0 (\ZZ)^{(d+1) \times (d+1)}$:
$$
\Ab *_2 \Bb = S_{\Ab} \Bb=\Ab\ast (\uparrow\Bb),
$$ 
that is
\begin{equation*}
  (\Ab *_2 \Bb)_j =\sum_{k \in \ZZ}\Ab_{j-2k}\Bb_k, \qquad j \in \ZZ.
\end{equation*}
Note that $*_2$ is neither commutative nor associative, due to which we
define iterated products as
$$
\Ab^{[n]} *_2 \dots *_2 \Ab^{[1]} = \Ab^{[n]} *_2 \left( \Ab^{[n-1]} *_2
  \dots *_2 \Ab^{[1]} \right),
$$
with $\Ab^{[j]} \in \ell (\ZZ)^{(d+1) \times (d+1)}$, $j=1,\dots,n$.

Starting with $\Ab *_2 \Bb *_2 \Cb = \left( \Ab *_2 \Bb \right) *_4
\Cb$, one can then easily prove by induction that
$$
\Ab^{[n]} *_2 \dots *_2 \Ab^{[1]} *_2 \Cb = \left( \Ab^{[n]} *_2 \dots
  *_2 \Ab^{[1]} \right) *_{2^n} \Cb,
$$
so that the 
application of  $m+1$ subdivision steps (\ref{eq:subdivop}) to an
initial sequence $\cb$ can be written, for $n,m \in \NN$, as
\begin{equation}\label{eq:conv}
 S_{\Ab^{[n+m]}}\cdots S_{\Ab^{[n]}}\cb=\Ab^{[n+m]} *_2 \cdots *_2 \Ab^{[n]} *_2\cb.
\end{equation}

Let $(\Ab^{[n]}: n\geq 0)$, be a sequence of finitely supported masks. A
\emph{level-dependent}  \emph{Hermite subdivision scheme} $S(\Ab^{[n]}
: n\geq 0)$ is the procedure of iteratively constructing vector sequences  by the rule
\begin{equation}\label{eq:hermitesubd}
 \Db^{n+1}\cb^{[n+1]}=S_{\Ab^{[n]}}\Db^{n}\cb^{[n]}, \quad n\in \NN,
\end{equation}
starting from an initial sequence $\cb^{[0]}$ of vector-valued  data. The
$k$-th component of $\cb^{[n]}$ is interpreted as the $k$-th derivative of a
function evaluated at the grid $2^{-n} \ZZ$.
In (\ref{eq:hermitesubd}), $\Db$ denotes the diagonal matrix
$\Db=\diag (1,\frac 12,\ldots, \frac 1{2^{d}})$ and the sequence
$\cb^{[n+1]}$ is related to the evaluation of function values and
consecutive derivatives on the dyadic grid $2^{-(n+1)}\ZZ$, where the powers
of $\Db$ in the iteration of (\ref{eq:hermitesubd}) correspond to a chain
rule for the derivatives.

If the same mask $\Ab$ is used at all levels of the subdivision
process, i.e., $\Ab^{[n]}=\Ab$, for $n \in \NN$, the associated Hermite
subdivision scheme is also called ``stationary'' sometimes. In the
following, unless explicitly specified, we always refer to the
level-dependent case.

A Hermite subdivision scheme as in (\ref{eq:hermitesubd}) is called
\emph{interpolatory} if $\cb^{[n+1]}_{2j}=\cb^{[n]}_j$, $j \in \ZZ$, for any $n \in \NN$. 
In this case, all the masks satisfy $\Ab^{[n]}_{2j}=\Db \delta_j$, $j
\in \ZZ$.

The scheme is said to be
\emph{$C^d$-convergent} for some $d\geq 1$, if for any input data
$\cb^{[0]} \in \ell_{\infty}(\ZZ)^{d+1}$ there exists a function
$\bPhi=[\phi_j]_{j=0}^d: \RR \to \RR^{d+1}$, such that the sequence
$\cb^{[n]}$ of refinements satisfies
\begin{equation*}
  \lim_{n \to \infty} \sup_{j \in \ZZ} | \cb^{[n]}_j - \bPhi
  \left(2^{-n}j \right)
  |_{\infty} = 0,
\end{equation*}
and where $\phi_0 \in C^d_u(\RR)$ as well as $\frac{d^j \phi_0}{dx^j}
=\phi_j$, $j=0,\ldots,d$, see \cite{dubuc05,merrien10}. Moreover,
convergence requests that the scheme is \emph{nontrivial}, i.e., that
there exists at least one $\cb^{[0]} \in \ell_{\infty}(\ZZ)^{d+1}$
such that the resulting limit function satisfies $\bPhi \neq 0$.

\subsection{Polynomial and exponential reproduction}
We are interested in Hermite subdivision schemes that reproduce polynomials and
exponentials, i.e., elements of the $(d+1)$-dimensional space 
\begin{equation*}
V_{p,\Lambda}=\operatorname{span}\{ 1,x,\ldots, x^p,
e^{\lambda_1x},\ldots,e^{\lambda_r x}\} 
\end{equation*}
where $\Lambda :=\{\lambda_1,\ldots, \lambda_r \}$ with  $\lambda_j \in
\CC \setminus \{ 0 \}$, $j=1,\ldots, r$ and $d=p+r$. Such schemes have
already
been studied in \cite{conti16,conti17a}, however restricted to \emph{pairs} of
exponential frequencies $\pm \lambda_k$ in $\Lambda$, due to technical
reasons. This restriction
is not needed here, the only requirement is that $\lambda_j \neq
\lambda_k$, $j \neq k$. In the sequel we will always assume that the
frequencies are all distinct.

Following \cite{conti16}, this reproduction property is formulated in
terms of the $V_{p,\Lambda}$-\emph{spectral condition}. In order to
give a definition, we need the following notation: For $f \in
C^d(\RR)$ we denote by $\vb_f: \RR \to \RR^{d+1}$ the vector valued
function 
\begin{equation*}
  \vb_f(x)=\left[f^{(j)}(x)\right]_{j=0}^d, \qquad x \in \RR,
\end{equation*}
and by
$\vb^{[n]}_f := \left. \vb_f \right|_{2^{-n} \ZZ} \in \ell(\ZZ)^{d+1} $
the vector-valued sequences with components
\begin{equation*}
 (\vb^{[n]}_f)_j=\vb_f(2^{-n}j) = \left[f^{(k)}(2^{-n}j)\right]_{k=0}^d,
\qquad j \in \ZZ.
\end{equation*}

\begin{Def}\label{def:spectralcond}
  A Hermite subdivision scheme $S( \Ab^{[n]} : n \ge 0 )$ is said to satisfy the
  \emph{$V_{p,\Lambda}$-spectral
    condition} if
  \begin{equation*}
    S_{\Ab^{[n]}} \Db^{n} \vb^{[n]}_f=\Db^{n+1} \vb^{[n+1]}_f, \qquad f
    \in V_{p,\Lambda}, \, n\in \NN.
  \end{equation*}
\end{Def}

\begin{Rem}
 The $V_{p,\Lambda}$-spectral condition is a generalization of the
 polynomial reproduction property of classical Hermite schemes
 introduced in \cite{dubuc09}, see also \cite{merrien10}. In
 \cite{dubuc09}, the reproduction of polynomials is called the
 \emph{spectral condition}. 
 Since $V_{d,\emptyset}=\Pi_d$, i.e., the space of polynomials of
 order up to $d$, the $V_{d,\emptyset}$-spectral condition is simply called 
 \emph{polynomial reproduction} or  \emph{spectral
   condition}.
\end{Rem}

\noindent
For $n,m \in \NN$, the $V_{p,\Lambda}$-spectral condition
implies
\begin{equation}\label{eq:general_spectralcond_level}
 S_{\Ab^{[n+m]}} \cdots
 S_{\Ab^{[n]}}\Db^n \vb^{[n]}_f=\Db^{m+n+1} \vb^{[n+m+1]}_f,
 \qquad f \in V_{p,\Lambda},
\end{equation}
which reduces to 
\begin{equation*}
 S^{m+1}_{\Ab}\Db^n \vb^{[n]}_f=\Db^{n+m+1} \vb^{[n+m+1]}_f,
 \qquad f \in \Pi_d, 
\end{equation*}
whenever $\Ab^{[n]} = \Ab$, $n \in \NN$.

We end the section by some remarks related to the possibility of
factorizing  the subdivision operator once it satisfies the the
exponential and polynomial preservation property. 
It is shown in \cite{conti16}
that  such a factorization can be given in terms of the so-called
\emph{cancellation operator}  $\cH^{[n]}: \ell(\ZZ)^{d+1} \to
\ell(\ZZ)^{d+1}$, which is a convolution operator $\cH^{[n]} \cb =
\Hb^{[n]} * \cb$ for some $\Hb^{[n]} \in \ell_0(\ZZ)^{(d+1)\times (d+1)}$
and $ \cb \in \ell(\ZZ)^{d+1}$ whose action is described by
\begin{equation*}
  0 = (\cH^{[n]} \vb^{[n]}_f)_j = \sum_{k\in \ZZ} \Hb^{[n]}_{j-k} 
  (\vb^{[n]}_f)_k, \qquad j \in \ZZ, \quad f\in V_{p,\Lambda}.
\end{equation*}
Recall from \cite{conti17a} that a cancellation operator $\cH$ for
$V_{p,\Lambda}$ is called \emph{minimal} if any other convolution
operator $\cH'$ with $\cH' \, V_{p,\Lambda} = 0$ has a factorizable
impulse response in the convolution algebra, i.e., 
$\Hb' = \Cb * \Hb$ for some
finitely supported matrix-valued sequence $\Cb$.
More specifically, the following theorem holds.
 
\begin{thm}\label{T:MiniFact}
  If, for $n \ge 0$, the subdivision operator $S_{\Ab^{[n]}}$  satisfies the
  $V_{p,\Lambda}$-spectral condition, then there exist a minimal
  cancellation operator $\cH^{[n]}$ and a finitely
  supported  mask $\Rb^{[n]}\in \ell_0(\ZZ)^{(d+1)\times (d+1)}$
  such that the factorization property
  \begin{equation}\label{eq:MiniFact}
    \cH^{[n+1]}S_{\Ab^{[n]}} = S_{\Rb^{[n]}}
    \cH^{[n]}
  \end{equation}
  holds true.
\end{thm}

\noindent
The structure of $\cH^{[n]}$ is given in \cite{conti16} for the
case where the exponentials in $V_{p,\Lambda}$ are associated to pairs
of frequencies $\pm\lambda_k$. In our more general setting, a similar
structure can be derived, as shown in the following lemma. For its
formulation, we use the notation 
$$
\Db(e^\Lambda)=\mbox{diag}\, (e^{\lambda_1},\ldots, e^{\lambda_r}),
\qquad 
\Wb_n(\Lambda)=\left[
  \begin{array}{ccc}
    1& \dots& 1\\
    \vdots& & \vdots\\
    \lambda_1^{n-1}& \dots & \lambda_r^{n-1} \\
  \end{array}\right]
$$
and keep in mind that both matrices are nonsingular since the
$\lambda_1,\dots,\lambda_r$ are nonzero and disjoint; the latter
guarantees that the \emph{Vandermonde matrix} $\Wb_n (\Lambda)$ is
invertible.

\begin{lem}
 The unique minimal cancellation operator for $V_{p,\Lambda}$ on level
 $n$ is given as
 $$
 \cH^{[n]}=\cH_{2^{-n}\Lambda},
 $$
 where $\cH_{\Lambda}$ is the convolution operator with associated symbol
 $$
 \Hb^*(z) := \Hb_\Lambda^* (z)
 = \left[
   \begin{array}{cc}
     z^{-1}\Ib+\Tb_0 & \Qb \\ \bZero & z^{-1}\Ib+\Rb_0
   \end{array}\right],
 $$
 defined by the scalar matrices 
 \begin{eqnarray*}
   \Tb_0&=&\left[ 
            \begin{array}{ccccc}   
              -1 &-1 & \cdots &-\frac 1{(p-1)!} &-\frac 1{p!}\\
              0 & -1 &\ddots & \vdots & \vdots\\
              \vdots &   &\ddots \\
                 &&& -1& -1\\[0.5em]
              0   & \dots & & 0 & -1
            \end{array}
                                  \right] \in \RR^{(p+1)\times (p+1)},\\
   \Qb&=&-\left( \Wb_{p+1}(\Lambda) \, \Db(e^\Lambda)+\Tb_0
          \Wb_{p+1}(\Lambda)\right)\Db(e^\Lambda)^{-p-1} \Wb_r(\Lambda)^{-1}\in \RR^{(p+1)\times r},
   \\
   \Rb_0&=&-\Wb_r(\Lambda) \, \Db(e^\Lambda)\, \Wb_r(\Lambda)^{-1}\in
            \RR^{r\times r}.
 \end{eqnarray*}
\end{lem}

\begin{proof}
  Proceeding like in \cite{conti16}, we denote by $\Tb_p^*(z)$ the
  symbol of the \emph{complete Taylor operator}
  \cite{merrien10}, and can obtain $\cH^{[n]}$ from the operator
  $\cH_{\Lambda}$
  with symbol
  $$
  \Hb^*(z)
  = \left[ \begin{array}{cc}
             \Tb_{p}^*(z) & \Qb \\ \bZero & \Rb^*(z)
           \end{array}\right]= \left[ \begin{array}{cc}
                                        z^{-1}\Ib+\Tb_0 & \Qb \\
                                        \bZero & z^{-1}\Ib+\Rb_0 
                                      \end{array}\right]
  $$
  for some $\Qb$ and $\Rb_0$, which must satisfy
  $$
  \Hb^* \left( e^{-\lambda_j} \right)\Db^n \,
  \left[ \lambda_j^{k}
  \right]_{k=0}^d
  = 0, \qquad j=1,\ldots,r.
  $$
  The conditions 
  $$
  \left[
    \begin{array}{c|c}
      e^{\lambda_j}\Ib+\Tb_0 & \Qb \\
      \hline
      \bZero & e^{\lambda_j}\Ib+\Rb_0
    \end{array}
  \right]\,
  \left[
    \begin{array}{l}
      \left[\lambda_j^k\right]_{k=0}^p \\[0.2cm]\hline
      \left[\lambda_j^k\right]_{k=p+1}^d
    \end{array}
  \right]=0, \qquad j=1,\ldots, r,
  $$
  can be written as
  $$
  \Wb_{p+1}(\Lambda)\, \Db(e^\Lambda)
  + \Tb_0\,\Wb_{p+1}(\Lambda)
  +\Qb\,\Wb_{r}(\Lambda)\, \Db(e^\Lambda)^{p+1}=0
  $$
  and
  $$
  \Wb_{r}(\Lambda)\, \Db(e^\Lambda)
  +\Rb_0 \,\Wb_{r}(\Lambda)=0,
  $$
  from which it follows that
  \begin{align*}
    \Qb&=
         -\left(\Wb_{p+1}(\Lambda)\,\Db(e^\Lambda)+\Tb_0
         \Wb_{p+1}(\Lambda)\right)
         \Db(e^\Lambda)^{-p-1} \Wb_r(\Lambda)^{-1},
    \\
    \Rb_0&=-\Wb_r(\Lambda) \, \Db(e^\Lambda)\, \Wb_r(\Lambda)^{-1}.
  \end{align*}
  The remaining arguments, in particular minimality, are as in
  \cite{conti17a}.
\end{proof}

\subsection{Basic limit functions and refinability}
Applying a $C^d$-convergent Hermite subdivision scheme $S(\Ab^{[n]} :
n \geq 0)$ to the input data $\delta \eb_j$, we
obtain, for each $j=0,\ldots d$, a vector consisting of a limit
function $\phi_j$ and all its derivatives. Together,
all such $\phi_j$, $j=0,\ldots,d$, give rise to the \emph{basic limit
  function}
\begin{equation*}
  \Fb=\left[\begin{array}{cccc}
    \phi_0  & \phi_1 & \ldots & \phi_d \\
    \phi_0' & \phi_1' & \ldots & \phi_d' \\
    \vdots & & & \\
    \phi_0^{(d)} & \phi_1^{(d)} & \ldots & \phi_d^{(d)}
  \end{array} \right].
\end{equation*}
In addition to the scheme $S(\Ab^{[n]} : n\geq 0)$ it is also useful
to consider the subdivision schemes
$S(\Ab^{[n+\ell]} : n\geq 0)$, $\ell \ge 0$, with the iteration
\begin{equation}\label{eq:allschemes}
  (\Db^{n+1}\cb^{[n+1]})_j=\sum_{k\in \ZZ}\Ab
  _{j-2k}^{[n+\ell]}\Db^n\cb^{[n]}_k, \qquad j\in\ZZ.
\end{equation}

\begin{Rem}\label{R:AsymptHermite}
  The assumption that $S(\Ab^{[n+\ell]} : n\geq 0)$, $\ell \ge 0$, is
  convergent is a standard one in level-dependent subdivision as it
  ensures the existence of a \emph{refinement equation} which we will
  give in the following Lemma~\ref{L:Refinement}. As shown in
  \cite{DynLevin95}, this property follows
  from the convergence of $S(\Ab^{[n]} : n\geq 0)$, provided that the
  scheme is \emph{asymptotically equivalent} to some $C^d$-convergent
  classical 
  Hermite subdivision scheme based on a mask $\Ab \in
  \ell_0 (\ZZ)^{(d+1) \times (d+1)}$. This property is defined as
  $$
  \sum_{n=0}^\infty \| S_{\Ab^{[n]}} - S_\Ab \|_\infty
  < \infty,
  $$
  cf. \cite{DynLevin95}.
  Also recall from \cite{conti16,conti17a} the fact that the
  $V_{p,\Lambda}$-preserving schemes built by \eqref{eq:MiniFact} with
  $\Rb^{[n]} := \Rb$, $n \ge 0$, for some $\Rb$
  are asymptotically equivalent to the Hermite scheme based on a
  Taylor factorization \cite{merrien10} with factor $\Rb$. Their
  limit mask is the one related to the Hermite subdivision scheme
  prserving $V_{p+r,\emptyset} = \Pi_d$. Note that this is the 
  ``Hermite analogy'' of the \emph{exponential B--spline} introduced
  in \cite{unser05}. 
\end{Rem}

\noindent
If we suppose that $S(\Ab^{[n+\ell]} : n\geq 0)$, $\ell \ge 0$, are
all convergent and
apply all such schemes to the same initial data 
$\cb^{[0]}=\delta\Ib_{d+1}$, we obtain
a \emph{sequence} of basic limit functions $(\Fb^{[\ell]} : \ell\geq
0)$, where $\Fb^{[0]}=\Fb$. Those basic limit functions are connected
by a \emph{level-dependent refinement equation} which we prove in the
next lemma for the sake of completeness.
We mention, however, that this refinement equation 
is already stated in \cite{cotronei17}, and well-known and popular
in the level-dependent non-Hermite \cite[Section 2.3]{dyn02}
as well as in the stationary Hermite \cite[Theorem 19]{dubuc05}
setting. The extension here is an adaption of technique given there.

\begin{lem}\label{L:Refinement}
  If $S( \Ab^{[n+\ell]} : n \ge 0)$, $\ell \ge 0$, are
  $C^d$-convergent Hermite subdivision schemes, then the associated
  sequence of basic limit functions
  $(\Fb^{[\ell]}: \ell \ge 0)$ satisfies 
  \begin{equation}\label{eq:refinement}
    \Fb^{[\ell]}=\sum_{k \in \ZZ}\Db^{-1}\Fb^{[\ell+1]}(2 \cdot -
    k)\Ab^{[\ell]}_k,
    \qquad \ell\in \NN.
  \end{equation}
\end{lem}

\begin{proof}
  We iterate the subdivision scheme on matrix valued data.
  For $\Cb^{[0]}=\delta\Ib_{d+1}$, the first iteration of
  (\ref{eq:allschemes}) yields for $\ell\in \NN$, that
  $\Db\Cb^{[1,\ell]}=\Ab^{[\ell]}$, hence
  \begin{eqnarray*}
  (\Db^{n+1}\Cb^{[n+1,\ell]})_j
  &=& \left( \Ab^{[\ell+n]} *_2 \cdots *_2 \Ab^{[\ell+1]} \right)
  *_{2^n} \Db \Cb^{[1,\ell]} \\
  &=& \left( \Ab^{[\ell+n]} *_2 \cdots *_2
    \Ab^{[\ell+1]} \right) *_{2^n} \Ab^{[\ell]}
  \end{eqnarray*}
  or, equivalently
  \begin{equation}
    \label{eq:LRefinementPf1}
    \Db^{n+1}\Cb^{[n+1,\ell]}
    =\Ab^{[\ell+n]} *_2 \cdots *_2 \Ab^{[\ell+1]} *_2 \Ab^{[\ell]}.  
  \end{equation}
  Interpreting (\ref{eq:LRefinementPf1}) as
  $$
  \left( \Ab^{[\ell+n]} *_2 \cdots *_2 \Ab^{[\ell+1]} \right) *_{2^n}
  \Ab^{[\ell]} = \left(\Db^{n}\Cb^{[n,\ell+1]}\right) *_{2^n} \Ab^{[\ell]},
  $$
  it follows that
  $$
  (\Db\Cb^{[n+1,\ell]})_j = \sum_{k\in \ZZ} \Cb^{[n,\ell+1 ]}_{j-2^nk}
  \Ab^{[\ell]}_k, \qquad j \in \ZZ.
  $$
  Let $x\in \RR$ and let $\left(j_n: n \in \NN \right)$ be a sequence of integers such that
  $j_n/2^n\to x$ as $n\to \infty$. Let   $ \Fb^{[n,\ell]}$ denote the
  piecewise linear matrix valued functions such that
  \begin{equation*}
    \Cb_{j_n}^{[n,\ell]}=\Fb^{[n,\ell]}\left(\frac {j_n}{2^n}\right), \qquad
    n,\ell \ge 0.
  \end{equation*}
  We have that
  \begin{equation*}
    \Db \Fb^{[n+1,\ell]}\left(\frac {j_n}{2^{n+1}}\right) =
    \sum_{k\in \ZZ}\Fb^{[n,\ell+1]}\left(\frac
      {{j_n}-2^nk}{2^{n}}\right) \, \Ab^{[\ell]}_k,
  \end{equation*}
  and the uniform convergence of $\Fb^{[n,\ell]}$ to $\Fb^{[\ell]}$ for
  $n \to \infty$ yields (\ref{eq:refinement}).
\end{proof}

\noindent
Note that in the case $\Ab^{[n]} = \Ab$, $n \ge 0$, there is only one
basic limit 
function $\Fb$ which satisfies the \emph{refinement equation}
\begin{equation}\label{eq:2_scale}
\Fb=\sum_{k\in\ZZ}\Db^{-1}\Fb(2\cdot - k)\Ab_k.
\end{equation}
Therefore Lemma~\ref{L:Refinement} is a generalization of 
\cite[Theorem 19]{dubuc05}, where (\ref{eq:2_scale}) can already be
found.

For interpolatory subdivision
schemes, the values of the basic limit functions $\Fb^{[\ell]}$ at the
dyadics $2^{-n}\ZZ$ are exactly the coefficients of the corresponding
schemes at level $n$ and $\Fb^{[\ell]}$ is \emph{cardinal}, that
is
\begin{equation}\label{eq:cardinality}
\Fb^{[\ell]}(k)=\delta_{k} \Ib_{d+1}, \qquad k \in \ZZ, \quad \ell\ge 0.
\end{equation}

By iteration of (\ref{eq:refinement}), it is easy to see that the sequence of basic limit functions of
a $C^d$-convergent Hermite subdivision scheme satisfies
\begin{equation}\label{eq:refinement_iteration}
  \Fb^{[\ell]} = \sum_{k \in \ZZ}
  \Db^{-m}\Fb^{[\ell+m]}(2^{m}\cdot-k)(\Ab^{[\ell+m-1]} *_2 \cdots *_2
  \Ab^{[\ell]})_k, \qquad \ell \ge 0, \, m \ge 1,
\end{equation}
which reduces to
\begin{equation*}
  \Fb=\sum_{k\in\ZZ}\Db^{-m}\Fb(2^m\cdot - k)\Ab_k^m, \qquad m\in \NN,
\end{equation*}
if all masks coincide. From (\ref{eq:cardinality}) and
(\ref{eq:refinement_iteration})
we get explicit representations for the 
basic limit functions of an interpolatory $C^d$-convergent Hermite
subdivision scheme at the integers, namely,
\begin{equation}\label{lem:leveln_card}
  \Fb^{[\ell]}(2^{-m}k)=\Db^{-m}(\Ab^{[\ell+m-1]} *_2 \cdots *_2
  \Ab^{[\ell]})_k, \qquad k \in \ZZ, \quad \ell \ge 0, \, m \geq 1,
\end{equation}
%
%
and
\begin{equation*}
  \Fb(2^{-m}k)=\Db^{-m}\Ab^m_k, \qquad k \in \ZZ, \quad m \geq 1,
\end{equation*}
respectively.

\section{Wavelets defined by interpolatory Hermite
  subdivision}\label{sec:MRA_Hermite}
In this section we briefly review the  construction of a
level-dependent MRA based on interpolatory Hermite subdivision which
was suggested in \cite{cotronei17}.

To that end, we start with a $C^d$-convergent \emph{interpolatory}
Hermite subdivision scheme $S(\Ab^{[n]} : n\geq 0)$. The sequence of
basic limit functions
$(\Fb^{[n]}: n\geq 0)$, more precisely their first rows
$\bPhi^{[n]}=[\phi^{[n]}_j]_{j=0}^d$, $n \ge 0$, span a
level-dependent MRA $(V_n: n \geq 0)$ for the space
$C^{d}_u(\RR)$. This means that the spaces $V_n$ are still nested but
the refinement property between $V_n$ and $V_{n+1}$ depends on $n$ as
it is now based on \eqref{eq:refinement}. Since the subdivision scheme
is interpolatory, the projection of $f \in C^{d}_u(\RR)$ onto $V_n$ is
given by the Hermite interpolant 
\begin{equation}\label{eq:projV}
{\cal P}_n f
=
\sum_{k\in \ZZ} \left(\bPhi^{[n]}\right)^T(2^n\cdot -k) \cb^{[n]}_k,
\end{equation}
with
\begin{equation*}
 \cb^{[n]}=\Db^n \vb_f^{[n]}.
\end{equation*}
The associated \emph{wavelet space} $W_n$ is the complement of $V_n$
in $V_{n+1}$. Taking into account that
\[
{\cal P}_{n+1} f={\cal P}_n f+ ({\cal P}_{n+1} - {\cal P}_n) f
={\cal P}_n f+{\cal Q}_n f,
\]
the projection ${\cal Q}_n f$ onto the wavelet space is given by
\begin{equation}\label{eq:def_Qn}
{\cal Q}_n f = {\cal P}_{n+1} f - {\cal P}_n f,
\end{equation}
and we set $W_n = \cQ_n V_{n+1}$.
It is shown in \cite{cotronei17} that
\begin{equation}\label{eq:Qn}
{\cal Q}_n f= \sum_{k \in \ZZ} (\bPhi^{[n+1]})^T(2^{n+1} \cdot -k)\db^{[n]}_k,
\end{equation}
where the wavelet coefficients are given by the
\emph{prediction-correction scheme}
\begin{equation}\label{eq:dn}
  \db^{[n]}=\cb^{[n+1]}-S_{\Ab^{[n]}}\cb^{[n]}.
\end{equation}
The \emph{interpolatory Hermite wavelet transform}  associates to
any $f\in C^d_u(\RR)$ a representation in terms of the vector-valued
\emph{decomposition} sequences:
$$
\cb^{[0]}, \, \db^{[0]}, \, \db^{[1]}, \, \db^{[2]},\ldots.
$$ 
Conversely, the coefficient sequence  connected to the projection
(\ref{eq:projV}) can be reconstructed as 
\begin{eqnarray*}
  \cb^{[n]} & = & \db^{[n-1]}+S_{\Ab^{[n-1]}}\cb^{[n-1]} \\
  &=&\db^{[n-1]}+S_{\Ab^{[n-1]}}\left(
      \db^{[n-2]}+S_{\Ab^{[n-2]}}\cb^{[n-2]}\right) \\
  &=&\db^{[n-1]}+
      S_{\Ab^{[n-1]}}\db^{[n-2]} +
            \cdots + S_{\Ab^{[n-1]}}\cdots S_{\Ab^{[0]}}\cb^{[0]}. 
\end{eqnarray*}
Incorporating  also the derivatives
of $f$ into \eqref{eq:projV} and \eqref{eq:Qn}, we find that
\begin{equation}\label{eq:projF}
\vb_{{\cal P}_n f}
=
\sum_{k\in \ZZ} \Db^{-n}\Fb^{[n]}(2^n\cdot -k) \left( \Db^n
\vb_f^{[n]} \right)_k
\end{equation}
and
\begin{align}\label{eq:wave_coeff_with_F}
\vb_{{\cal Q}_n f}
=
\sum_{k\in \ZZ} \Db^{-n-1}\Fb^{[n+1]}(2^{n+1}\cdot -k)\db^{[n]}_k. 
\end{align}
From (\ref{eq:def_Qn}) it also follows that
\begin{align}\label{eq:vector_QP}
\vb_{{\cal Q}_n f}=\vb_{{\cal P}_{n+1} f}-\vb_{{\cal P}_n f}.
\end{align}
We recall that in the classical (non-Hermite) situation, the
reproduction  of polynomials up to the degree $d$ by the subdivision
scheme implies polynomial vanishing moments for the
wavelets. In our setting, polynomial reproduction is replaced by the
$V_{p,\Lambda}$-spectral condition from Definition \ref{def:spectralcond},
which is a condition on the 
function $f\in V_{d,\Lambda}$ and all its derivatives up to order
$d$.
This has also consequences for the corresponding MRA, resulting 
in a \emph{$V_{p,\Lambda}$ vanishing moment} property, that is, the
wavelet coefficients (\ref{eq:dn}) connected to any $f\in
V_{p,\Lambda}$ are all zero.
A first property of the  projections, which is useful for proving
the result on the decay of the wavelet coefficients in 
Section~\ref{sec:decay}, is stated and proved in the following lemma.
\begin{lem}\label{le:vp}
  Let $S(\Ab^{[n]} : n\geq 0)$ be an interpolatory $C^d$-convergent Hermite
  subdivision scheme satisfying the $V_{p,\Lambda}$-spectral
  condition. Then we have
  \begin{equation}\label{eq:Pnprojvb}
    \vb_{{\cal P}_nf}=\vb_f, \qquad f \in V_{p,\Lambda}, \quad n\in \NN.
  \end{equation}
\end{lem}

\begin{proof}
  Since $f$ and all its derivatives are continuous by assumption, and
  since the dyadic points are dense in $\RR$, it is sufficient to
  verify \eqref{eq:Pnprojvb} for dyadic points of the form $2^{-m}k$, $k \in
  \ZZ$, $m \in \NN$. We fix $n \in \NN$ and distinguish between the
  cases $m=n, m<n$ and $m>n$. 
	
  If $m=n$, then $\vb_{{\cal P}_nf}(2^{-n}k)=\vb_{f}(2^{-n}k)$ follows
  directly from \eqref{eq:cardinality} and \eqref{eq:projF}, since the
  scheme is interpolatory.
  In the case that $r := n-m$ is positive, 
  get
  \begin{align*}
    \Db^n\vb_{{\cal P}_nf}
    & (2^{-m} k)
      = \sum_{\ell \in \ZZ} \Fb^{[n]}(2^n2^{-m} k -\ell)\Db^n
      \vb_f(2^{-n}\ell)\\ 
    &= \sum_{\ell \in \ZZ} \Fb^{[n]}(2^{r} k -\ell) \left( \Db^n
      \vb_f^{[n]} \right)_\ell\\
    &= \sum_{\ell \in \ZZ} \Fb^{[n]} \left( 2^{-r} (2^{2r}k
      -2^{r}\ell) \right) \left( \Db^n
      \vb_f^{[n]} \right)_\ell \\ 
    &= \Db^{-r} \left( \left( \Ab^{[n+r-1]} *_2 \cdots *_2
      \Ab^{[n]} \right) *_{2^{r}} \Db^n \vb_f^{[n]}
      \right)_{2^{2r} k} \\  
    &= \Db^{-r} \left( S_{\Ab^{[n+r-1]}}\cdots S_{\Ab^{[n]}}
      \Db^n \vb_{f}^{[n]} \right)_{2^{2r}k}\\
    &= \Db^{-r} \Db^{r-1+n+1} \left( \vb_f^{[r-1+n+1]} \right)_{2^{2r}k}
      = \left( \Db^{n}\vb_{f}^{[n+r]} \right)_{2^{2r}k}\\
    &= \Db^{n}\vb_{f}(2^{-n-r}2^{2r}k) = \Db^{n}\vb_{f}(2^{-m}k),
  \end{align*}
  and in the final case $r<0$, we set $s = -r>0$ and compute likewise
  \begin{align*}
    \Db^n \vb_{{\cal P}_nf  }
    & (2^{-m} k)=\sum_{\ell \in \ZZ} \Fb^{[n]}(2^{-s} k -\ell) \left( \Db^n
      \vb_f^{[n]} \right) \\ 
    &= \Db^{-s} \left( (\Ab^{[n+s-1]} *_2 \cdots *_2 \Ab^{[n]})
      *_{2^{s}}\Db^n \vb_{f}^{[n]} \right)_k \\
      & = \Db^{-s} \Db^{s-1+n+1} \left( \vb_f^{[s-1+n+1]} \right)_k 
        =\left( \Db^{n} \vb_f^{[m]} \right)_k
      =\Db^{n} \vb_{f}(2^{-m}k),
  \end{align*}
  which completes the proof.
\end{proof}

\section{Taylor formula with exponentials}\label{sec:taylor}
As already mentioned, estimates for the decay rate of the wavelet
coefficients are usually based on a local Taylor polynomial
approximation and the polynomial reproduction property of the
operator. When dealing not only
with polynomial but also exponential vanishing moments, 
a more general tool is needed to fully explore the approximation power
of the space $V_{p,\Lambda}$.

In this section we derive such a generalized Taylor formula, by
means of elements in the space $V_{p,\Lambda}$, namely an approximation of the form
\begin{equation}
\label{eq:ExpoTaylorGeneral}
f(x+h) \approx T_{p,\Lambda}f (x,h) := \sum_{j=0}^p \frac{f^{(j)}
	(x)}{j!} h^j + \sum_{k=1}^r \mu_k (f) \, e^{\lambda_k h}.  
\end{equation}
We show that for proper functionals $\mu_k$, such an expression
can obtain an error of the order $h^{d+1}$ for a function $f \in C^d
(\RR)$ where $d = p+r$.
In Lemma~\ref{lem:TdLambdApprox} we  give the appropriate choice
for $\mu_1,\ldots,\mu_k$, such that the same approximation rate as
that of the usual \emph{Taylor operator}
\begin{equation*}
T_d f (x,h) = \sum_{j=0}^d \frac{f^{(j)} (x)}{j!} h^j
\end{equation*}
is obtained.

In our construction, we use the nonsingular Vandermonde matrix
$$
\Vb_d := \left[
\begin{array}{cccc|cccc}
1 & 0 & \dots & 0 & 1 & \dots & 1 \\
0 & 1 & \dots & 0 & \lambda_1 & \dots & \lambda_r \\
& & \ddots & \vdots & \vdots & \ddots &
\vdots \\
& & & p! & \lambda_1^p & \dots & \lambda_r^p \\
\hline
& & & & \lambda_1^{p+1} & \dots & \lambda_r^{p+1} \\
& & & & \vdots & \ddots & \vdots \\
& & & & \lambda_1^{d} & \dots & \lambda_r^{d} \\
\end{array}
\right] =: \left[
\begin{array}{cc}
\Xb & \Yb \\ & \Zb \\
\end{array}
\right]
\in \RR^{(d+1) \times (d+1)},
$$
corresponding to the poised, i.e., uniquely solvable, \emph{Hermite
  interpolation problem} at the $(d+1)$-fold point $0$ and at
$\Lambda$. Note that the square matrix $\Zb\in \RR^{r \times r}$ satisfies
$$
\Zb = \left[
\begin{array}{ccc}
1 & \dots & 1 \\
\vdots & \ddots & \vdots \\
\lambda_1^{r-1} & \dots & \lambda_r^{r-1}
\end{array}
\right] \left[
\begin{array}{ccc}
\lambda_1^{p+1} && \\
& \ddots & \\
& & \lambda_r^{p+1} \\
\end{array}
\right] 
$$
and, therefore, is the product of a Vandermonde matrix and a nonzero diagonal matrix,
hence invertible. The inverse of $\Vb_d$ is then
\begin{equation}\label{eq:inverse_V}
\Vb_d^{-1} = \left[
\begin{array}{cc}
\Xb^{-1} & -\Xb^{-1} \Yb \Zb^{-1} \\ & \Zb^{-1}
\end{array}
\right].
\end{equation}

\begin{lem}\label{lem:TdLambdApprox}
  If $f \in C^{d}(\RR)$ then
  \begin{equation}
    \label{eq:ExpoTaylorDef}
    T_{p,\Lambda} f (x,h) := \left[ 1,\dots,h^p,e^{\lambda_1 h}, \cdots
      e^{\lambda_r h} \right] \Vb_d^{-1} 
    \vb_f(x)    
  \end{equation}
  satisfies, for any $R < 1$,
  \begin{equation}
    \label{eq:ExpoTaylorError}
    \left| T_{p,\Lambda} f (x,h) - T_d f (x,h) \right| \le
    C_{\Lambda,R} | \vb_f (x) |_2 \, |h|^{d+1}, \qquad |h| \le R,
  \end{equation}
  where the constant $C_{\Lambda,R}$ depends on $\Lambda$ and $R$ only.
\end{lem}

\begin{proof}
  Recall from \cite{conti17a} that
  \begin{equation}
    \label{eq:BasisConvert}
    \left[
      \begin{array}{l}
	\left[x^j \right]_{j=0}^p \\[0.1cm] \hline \\[-0.3cm] \left[ e^{\lambda_j x} \right]_{j=1}^r
      \end{array}
    \right] = 
    \left[ \Vb_d^T \,|\, \Mb_\Lambda \Vb_d^T \,| \, \Mb_\Lambda^2 \Vb_d^T \,|\, \dots
    \right] 
    \left[ \frac{x^j}{j!} \right]_{j \in \NN}
  \end{equation}
  with
  $$
  \Mb_\Lambda = \left[
    \begin{array}{cccc}
      \mathbf{0}_{(p+1) \times (p+1)}
      &&& \\
      & \lambda_1^{d+1} & & \\
      & & \ddots & \\
      & & & \lambda_r^{d+1} \\
    \end{array}
  \right].
  $$
  Hence,
  \begin{eqnarray*}
    \lefteqn{\left( \Vb_d^{-1} \left[
    f^{(j)}(x)
    \right]_{j=0}^d \right)^T \left[
    \begin{array}{l}
      \left[ h^j\right]_{j=0}^p \\[0.2cm] \hline \\[-0.3cm] \left[ e^{\lambda_j h} \right]_{j=1}^r \\
    \end{array}
    \right] } \\
    & = & \left(\left[
          f^{(j)}(x)
    \right]_{j=0}^d\right)^T \left[ \Ib \,|\, \Vb_d^{-T} \Mb_\Lambda \Vb_d^T \,| \,
    \Vb_d^{-T} \Mb_\Lambda^2
    \Vb_d^T \,|\, \dots 
    \right] 
    \left[ \frac{h^j}{j!} \right]_{j \in \NN}  \\
    & = & \sum_{j=0}^d \frac{f^{(j)} (x)}{j!} h^j +  \left(\left[
    f^{(j)}(x)
    \right]_{j=0}^d \right)^T
    \sum_{k=1}^\infty \sum_{j=0}^d \Vb_d^{-T} \Mb_\Lambda^k \Vb_d^T \eb_j \,
    \frac{h^{k(d+1)+j}}{\left( k(d+1) + j \right)!} \\
    & = & T_d f(x,h) + \vb_f (x)^T \sum_{k=1}^\infty \sum_{j=0}^d
          \Vb_d^{-T} \Mb_\Lambda^k \Vb_d^T \eb_j \,
          \frac{h^{k(d+1)+j}}{\left( k(d+1) + j \right)!}.
  \end{eqnarray*}
  The spectral radius in the sum satisfies
  $$
  \rho \left( \Vb_d^{-T} \Mb_\Lambda
    \Vb_d^T \right) = \rho \left( \Mb_\Lambda \right) =
  \max_{j=1,\dots,r} \left| \lambda_j \right| =: \rho,
  $$
  and, since 	 $\Vb_d^{-T} \Mb_\Lambda^k \Vb_d^T = \left( \Vb_d^{-T} \Mb_\Lambda 
    \Vb_d^T \right)^k$, 
  there exists for any $\varepsilon > 0$ a constant $C > 0$, depending
  on $\Lambda$ and $\varepsilon$ such that
  \begin{equation}\label{eq:MatrixNormApprox}
    \left| \Vb_d^{-T} \Mb_\Lambda^k \Vb_d^T \right|_2 \le C \, (\rho +
    \varepsilon)^k, \qquad k \ge 0.
  \end{equation}
  Hence,
  \begin{eqnarray*}
    \lefteqn{\left| T_{p,\Lambda} f (x,h) - T_d f (x,h) \right|} \\
    & \le & \left| \vb_f (x) \right|_2 \sum_{k=1}^\infty \sum_{j=0}^d \left|
            \Vb_d^{-T} \Mb_\Lambda^k \Vb_d^T \right|_2 
            \frac{|h|^{k(d+1)+j}}{\left( k(d+1) + j \right)!} \\
    & \le & C \left| \vb_f (x) \right|_2 \sum_{k=1}^\infty (\rho +
            \varepsilon)^k |h|^{k(d+1)} \sum_{j=0}^d \frac{|h|^{j}}{\left(
			k(d+1) + j \right)!} \\
    & = & C \left| \vb_f (x) \right|_2 (\rho + \varepsilon) \,
          |h|^{d+1} \sum_{k=0}^\infty \left( ( \rho +
          \varepsilon) |h|^{d+1} \right)^k \sum_{j=0}^d \frac{|h|^{j}}{\left(
          (k+1)(d+1) + j \right)!}.
  \end{eqnarray*}
  Since for $|h| \le R < 1$,
  $$
  \sum_{j=0}^d \frac{|h|^{j}}{\left( (k+1)(d+1) + j \right)!} \le
  \frac1{\left( (k+1)(d+1) \right)!} \sum_{j=0}^d |h|^j \le \frac1{k!}
  \frac{1-R^{d+1}}{1-R},
  $$
  we can conclude that
  \begin{eqnarray*}
    \lefteqn{\sum_{k=0}^\infty \left( ( \rho + \varepsilon) |h|^{d+1} \right)^k
    \sum_{j=0}^d \frac{|h|^{j}}{\left((k+1)(d+1) + j \right)!}} \\
    & \le & \frac{1-R^{d+1}}{1-R} \sum_{k=0}^\infty \frac{\left( ( \rho +
            \varepsilon) |h|^{d+1} \right)^k}{k!}
            = \frac{1-R^{d+1}}{1-R} e^{( \rho +
            \varepsilon) |h|^{d+1}} \le \frac{1}{1-R} e^{\rho +
            \varepsilon}
  \end{eqnarray*}
  which is a constant that depends only on $\Lambda$ and $R$. Hence,
  $$
  \left| T_{p,\Lambda} f (x,h) - T_d f (x,h) \right| \le C \left| \vb_f
  (x) \right|_2 \frac{1}{1-R} e^{\rho + \varepsilon} \,
  \left( \rho + \epsilon \right) |h|^{d+1}
  $$
  uniformly in $|h| \le R < 1$ and if we combine all these numbers into a
  single constant, we get (\ref{eq:ExpoTaylorError}).
\end{proof}

\begin{Rem}\label{R:hhochd+1}
  We find it worthwhile to note that even when $f$ is only $d$-times
  continuously differentiable, the deviation between $T_{p,\Lambda}
  f(x,h)$ and $T_d f (x,h)$, which both use derivatives up to order
  $d$, is of order $h^{d+1}$ for small $h$ and not only of order
  $h^d$. In other words, the difference between the operators is 
  smaller than their approximation to $f$ due to which they can indeed
  be considered equivalent.
\end{Rem}

\noindent
Since
$$
\Vb_d^{-1} = \left[
\begin{array}{cccccc}
1& & & * & \dots & * \\
& \ddots & & \vdots & \ddots & \vdots\\
& & 1/p! & * & \dots & * \\
& & & * & \dots & * \\
& & & \vdots & \ddots & \vdots\\
& & & * & \dots & * \\
\end{array}
\right],
$$
the initial polynomial part of $T_{d,\Lambda} f$ is indeed the usual
Taylor polynomial $T_p f$
which means that the operator defined in (\ref{eq:ExpoTaylorDef}) is
indeed of the form (\ref{eq:ExpoTaylorGeneral}), the $\mu_j$ being
defined by the inverse of the Vandermonde matrix.

\begin{cor}\label{F:TdLambdApprox2n}
  With the setting of Lemma~\ref{lem:TdLambdApprox}, we have that
  \begin{equation*}
    \left| T_{p,2^{-n} \Lambda} f (x,h) - T_d f (x,h) \right| \le 2^{-n(d+1)}
    C_{\Lambda,R} \left| \vb_f (x) \right|_2 \, |h|^{d+1}, \quad |h| \le R.
  \end{equation*}
\end{cor}

\begin{proof}
  We only have to notice that $\Mb_{2^{-n} \Lambda} = 2^{-n(d+1)}
  \Mb_{\Lambda}$, hence, like in the preceding proof,
  \begin{eqnarray*}
  \lefteqn{\left| T_{d,2^{-n} \Lambda} f (x,h) - T_d f (x,h) \right|}\\
  &\le& \left| \vb_f (x) \right|_2 \sum_{k=1}^\infty 2^{-n(d+1)k} \, \sum_{j=0}^d \left|
    \Vb_d^{-T} \Mb_\Lambda^k \Vb_d^T \right|_2 
  \frac{|h|^{k(d+1)+j}}{\left( k(d+1) + j \right)!}
 \end{eqnarray*}
  and the rest follows as above with an even smaller constant.
\end{proof}

\noindent
Like in the classical Taylor formula we also have results for the
derivatives of the approximant from $V_{p,\Lambda}$.

\begin{lem}
  For $f \in C^d(\RR)$ and $x \in \RR$ one has
  \begin{equation}\label{eq:deriv1}
    (T_{p,\Lambda}f)^{(j)}(x,\cdot)=T_{p-j,\Lambda}f^{(j)}(x,\cdot),
    \qquad j=0,\ldots,p,
  \end{equation}
  and
  \begin{align}\label{eq:deriv2}
    \lefteqn{\frac{d^j}{dh^j} (T_{p,\Lambda}f) (x,h)}\nonumber\\
    &=
    \left[e^{\lambda_j h},\ldots, e^{\lambda_r h} \right]  
    \left[ \begin{array}{ccc}
             \lambda_1^{p+1-j} & \cdots & \lambda_r^{p+1-j}\\
             \vdots & \ddots & \vdots \\
             \lambda_1^{-1} & \cdots & \lambda_r^{-1}\\ \hline
             1 & & 1 \\
	\vdots & \ddots & \vdots \\
             \lambda_1^{d-j} & \cdots & \lambda_r^{d-j}
           \end{array}
	\right]^{-1}
	\left[
          \begin{array}{l}
	    \left[ f^{(k)}(x) \right]_{k=p+1}^{j-1}\\[0.2cm]
            \hline \\[-0.3cm]
            \left[ f^{(k)}(x) \right]_{k=j}^{d}
          \end{array}
	\right],
      \end{align}
      for $j=p+1,\ldots,d$.
\end{lem}

\begin{proof}
  For $j=0,\ldots,p$ we have
  \begin{align*}
     \lefteqn{(T_{p,\Lambda}f)^{(j)}(x,h)} \\
     &=
      \left[0,\ldots,0,j!,\ldots,\frac{p!}{(p-j)!}h^{p-j},\lambda_1^{j}e^{\lambda_1
      h}, \ldots, \lambda_r^{j}e^{\lambda_r h} 
      \right]\Vb_d^{-1}\vb_f(x) \\
    &= \left[ 0, \ldots,0, 1,\ldots, h^{p-j}, e^{\lambda_1 h},
      \ldots e^{\lambda_r h} \right] 
      \diag \left(1, \ldots,\frac{p!}{(p-j)!}, \lambda_1^{j},
      \ldots, \lambda_r^{j}\right) 
                              \Vb_d^{-1}\vb_f(x).        
  \end{align*}
  Furthermore,
  \begin{align*}
    \lefteqn{\diag \left(1, \ldots,\frac{p!}{(p-j)!}, \lambda_1^{j}, \ldots,
            \lambda_r^{j}\right) \cdot \Vb_d^{-1}} \\
          &=\left( \Vb_d \cdot \diag \left(1,
            \ldots,\frac{(p-j)!}{p!}, \lambda_1^{-j}, \ldots,
            \lambda_r^{-j}\right)\right)^{-1} \\
          &=  \left[
            \begin{array}{ccc|cccccc}
              1 & &  & & & & \lambda_1^{-j} & \dots & \lambda_r^{-j} \\
                & \ddots & &  & & & \vdots & \ddots  & \vdots \\
                & & 1 & & & & \lambda_1^{-1} & \dots & \lambda_r^{-1} \\
              \hline
                & &  & 1 &  &  & 1 & \dots & 1 \\
                & & & & \ddots & & \vdots & \ddots & \vdots \\
                & & & & & (p-j)! & \lambda_1^{p-j} & \dots & \lambda_r^{p-j} \\
                & & & & & & \vdots & \ddots & \vdots \\
                & & & & & & \lambda_1^{d-j} & \dots & \lambda_r^{d-j} \\
            \end{array}
    \right]^{-1} 
    = \left[ \begin{array}{c|c} 
               \Ib & \Wb\\ \hline
                 & \Vb_{p-j}
             \end{array}
                   \right]^{-1}.
  \end{align*}
  By (\ref{eq:inverse_V}) we obtain
  \begin{align*}
    (T_{p,\Lambda}f)^{(j)}(x,h)
    &=\left[ 0, \ldots,0, 1,\ldots,
      h^{p-j}, e^{\lambda_1 h}, \ldots e^{\lambda_r h} \right]
      \left[ \begin{array}{c|c} 
               \Ib & -\Wb \Vb_{p-j}^{-1} \\ \hline
                   & \Vb_{p-j}^{-1} \\
             \end{array}
    \right] \vb_f(x) \\
    &= \left[ 1,\ldots, h^{p-j}, e^{\lambda_1 h}, \ldots e^{\lambda_r
      h} \right] \Vb_{p-j}^{-1} \left[ f^{(k)}(x)
    \right]_{k=j}^d \\ 
    &= T_{p-j,\Lambda}f^{(j)}(x,h).
  \end{align*}
  For $j=p+1,\ldots, d$, the polynomial part vanishes
  completely. Equation \eqref{eq:deriv2} then follows from similar
  computations as the ones just carried out. 
\end{proof}

Using the preceding result, we obtain, in analogy with
Lemma~\ref{lem:TdLambdApprox}, similar error bounds between the derivatives
of the Taylor polynomial and (\ref{eq:ExpoTaylorDef}).

\begin{cor}\label{cor:Taylor_Deriv}
  Let $f \in C^d(\RR)$. For any $R<1$ and $j=0,\ldots,d$ we have
  \begin{equation*}
    |(T_{p,\Lambda}f)^{(j)}(x,h)-T_{d-j}f^{(j)}(x,h)|\leq
    C_{\Lambda,R}  \left| \vb_{f^{(k_j)}}(x) \right|_2 
    |h|^{d+1-j}, \quad |h| \leq R,
  \end{equation*}
  where $k_j= \min \{ j,p+1 \}$.
\end{cor}

\begin{proof} 
  For $j=0,\ldots, p$ the statement follows by combining
  \eqref{eq:ExpoTaylorError} and \eqref{eq:deriv1}. The proof of the
  case $j=p+1,\ldots,d$, does not follow directly from
  \eqref{eq:ExpoTaylorError} and \eqref{eq:deriv2}, because 
  of the special form of (\ref{eq:deriv2}). Still the proof is very
  similar to that of Lemma~\ref{lem:TdLambdApprox}. Let
  $j=p+1,\ldots,d$ and define
  \begin{equation*}
    \Wb =
    \left[
      \begin{array}{ccc}
        \lambda_1^{p+1-j} & \cdots & \lambda_r^{p+1-j}\\
        \vdots & \ddots & \vdots \\
        \lambda_1^{-1} & \cdots & \lambda_r^{-1}\\ \hline
        1 & & 1 \\
        \vdots & \ddots & \vdots \\
        \lambda_1^{d-j} & \cdots & \lambda_r^{d-j} \\
      \end{array}
    \right]\in \RR^{r\times r},
  \end{equation*}
  where the upper part of $\Wb$ is of dimension $(j-p-1) \times r$,
  and the lower part a $(d-j+1) \times r$ matrix.
  Furthermore, we introduce the $r \times r$ matrix
  \begin{equation*}
    \Nb_{\Lambda}=
    \left[
      \begin{array}{ccc}
        \lambda_1^{d-j+1} & & \\
                          & \ddots & \\
                          & & \lambda_r^{d-j+1}
      \end{array}
    \right].
  \end{equation*}
  Similar to (\ref{eq:BasisConvert}) we have
  \begin{equation*}
    \left[
      e^{\lambda_k x}
    \right]_{k=1}^r = 
    \left[ \Wb^T \,|\, \Nb_\Lambda \Wb^T \,| \, \Nb_\Lambda^2 \Wb^T \,|\, \dots
    \right] 
    \left[
      \begin{array}{c} \bZero_{j-p-1} \\ \hline \\[-0.3cm]  \left[ \dfrac{x^k}{k!}\right]_{k \in \NN}
      \end{array}
    \right].
  \end{equation*}  
  Therefore, by (\ref{eq:deriv2})
  \begin{align*}
    \lefteqn{(T_{p,\Lambda}f)^{(j)}(x,h)
    =\left(\left[ f^{(k)}(x)\right]_{k=p+1}^d\right)^T \Wb^{-T}
    \left[e^{\lambda_k h}
        \right]_{k=1}^r} \\ 
    &=\left(\left[ f^{(k)}(x)\right]_{k=p+1}^d\right)^T
        \left[ \Ib_{r} \,|\,\Wb^{-T}\Nb_\Lambda \Wb^T \,| \,
        \Wb^{-T}\Nb_\Lambda^2 \Wb^T \,|\,\dots \right]
    \left[
      \begin{array}{c} \bZero_{j-p-1} \\ \hline \\[-0.3cm]
      \left[ \dfrac{x^k}{k!}\right]_{k \in \NN}
      \end{array}
    \right].
    \\ 
    &= T_{d-j}f^{(j)}(x,h)+\left(\left[ f^{(k)}(x) 
                                    \right]_{k=p+1}^d\right)^T
    \sum_{\ell=1}^{\infty}\sum_{j=1}^r
    \Wb^{-T}\Nb_\Lambda^{\ell} \Wb^T \eb_j
    \frac{h^{d-j+(\ell-1)r+j}}{(d-j+(\ell-1)r+j)!}. 
  \end{align*}
  Using the same arguments as in (\ref{eq:MatrixNormApprox}) we have
  for $|h| \le R < 1$ that
  \begin{align*}
    | (T_{p,\Lambda}f)^{(j)}
    & (x,h)-T_{d-j}f^{(j)}(x,h) |  \\
    &  \leq \left|\vb_{f^{(p+1)}}(x)\right|_2\sum_{\ell=1}^{\infty}\sum_{j=1}^r
      \left|\Wb^{-T}\Nb_\Lambda^{\ell} \Wb^T\right|_2
      \frac{|h|^{d-j+(\ell-1)r+j}}{(d-j+(\ell-1)r+j)!} \\
    &  \leq \left|\vb_{f^{(p+1)}}(x)\right|_2 |h|^{d-j+1}
      \sum_{\ell=0}^{\infty}\sum_{j=0}^{r-1} (\rho+\epsilon)^{\ell+1}
      \frac{|h|^{\ell r+j}}{(d-j+\ell r+j+1)!}\\ 
    &  \leq \left|\vb_{f^{(p+1)}}(x)\right|_2 |h|^{d-j+1} (\rho+\epsilon)
      e^{R^r(\rho+\epsilon)} \frac{1-R^r}{1-R} .
  \end{align*}
  This gives
  \begin{equation*}
    |(T_{p,\Lambda}f)^{(j)}(x,h)-T_{d-j}f^{(j)}(x,h)| \leq
    C_{\Lambda,R} \, \left|\vb_{f^{(p+1)}}(x)\right|_2 \, |h|^{d-j+1} \\ 
  \end{equation*}
  as claimed.
\end{proof}

\noindent	
In the case of functions in $C^d_u(\RR)$, Corollary~\ref{cor:Taylor_Deriv} immediately gives a bound independent of $x$.

\begin{cor}\label{cor:Taylor_Little_o}
  For $f \in C^d_u(\RR)$, $j=0,\ldots d$ and any $R< 1$
  \begin{equation*}
    |(T_{p,\Lambda}f)^{(j)}(x,h)-T_{d-j}f^{(j)}(x,h)|\leq
    C_{\Lambda,R,f} \, |h|^{d-j+1}, \qquad |h|< R. 
  \end{equation*}
\end{cor}

\begin{proof}
  Since $f \in C^d_u(\RR)$ we have for $x \in \RR$ and $j=0,\ldots,d$ that
  $$
  \left|\vb_{f^{(j)}}(x)\right|_2
  \leq \left|\vb_{f^{(j)}}(x)\right|_{\infty} = \max_{k=j,\ldots, d} |f^{(k)}(x)| 
  \leq \max_{k=0,\ldots, d} \sup_{y \in \RR} |f^{(k)}(y)| =  \|f\|_{d,\infty}.
  $$
  Therefore $\left|\vb_{f^{(j)}}(x)\right|_2$ is bounded by a finite constant
  independent of $x \in \RR$ and for $j=0,\ldots,d$.
\end{proof}

Finally we are able to determine the asymptotic behavior of the
remainder term for any function (and derivatives) approximated with
the generalized Taylor formula.

\begin{lem}\label{lem:errorGeneralizedTaylor}
  Let $f \in C^d_u(\RR)$. Then for any $R<1$ we have
  \begin{equation*}
   \left| \left[ h^j \left(f^{(j)}(x+h)-
      (T_{p,\Lambda}f)^{(j)}(x,h)\right)\right]_{j=0}^{d}
    \right|_{\infty} \leq C_{\Lambda,R,f}\,h^{d+1}, \qquad |h|<R.
  \end{equation*}
\end{lem}

\begin{proof}
  It is well-known that the derivatives in the usual Taylor formula
  of $f$ are given by
  \[
  \left(T_df\right)^{(j)}(x,h)=\left(T_{d-j}f^{(j)}\right)(x,h), \qquad j=0,\ldots,d,
  \]
  with the remainder terms satisfying
  \begin{equation}\label{eq:taylor_Hd}
   |h^j\left( f^{(j)} (x+h) - (T_{d}f)^{(j)}(x,h)\right)|\leq C_{\Lambda,R,f}\,h^{d+1},\quad |h|<R, \, j=0,\ldots,d.
  \end{equation}
  We then have 
  \begin{eqnarray*}
	\lefteqn{\left|\left[ h^j \left(f^{(j)}(x+h)- (T_{p,\Lambda}f)^{(j)}(x,h)\right)\right]_{j=0}^{d} \right|_{\infty}} \\
	& =&\, \max_{j=0,\ldots, d} |h^j (f^{(j)}(x+h)- (T_{p,\Lambda}f)^{(j)}(x,h)) | \\
	& \leq& \max_{j=0,\ldots, d}\left( |h^j (f^{(j)}(x+h)- (T_{d}f)^{(j)}(x,h))|\right.\\
	& &+\left.|h^{j}((T_{d}f)^{(j)}(x,h)- (T_{p,\Lambda}f)^{(j)}(x,h))|\right).
	\end{eqnarray*}
	Therefore, in virtue of (\ref{eq:taylor_Hd}) and Corollary~\ref{cor:Taylor_Little_o}, it follows that:
	\begin{equation*}
	\left|\left[ h^j (f^{(j)}(x+h)- (T_{p,\Lambda}f)^{(j)}(x,h))\right]_{j=0}^{d} \right|_{\infty}\leq C_{\Lambda,R,f}\,h^{d+1}, \qquad |h|<R. \qedhere
	\end{equation*}
\end{proof}

\section{Decay of wavelet coefficients}\label{sec:decay}
In this section we give our main result, namely we prove that the wavelet coefficients associated to an interpolatory MRA generated
by a level-dependent Hermite subdivision scheme satisfying the
$V_{p,\Lambda}$-spectral condition  decrease of a certain
well-defined order as the scale increases and we give estimates of such a decay.
To our knowledge this has never been investigated in the Hermite
setting, even in the case of only polynomial reproduction.
Our proof exploits the generalized Taylor formula associated to a function in $C^d_u(\RR)$ given in the previous section.

We start with some remarks concerning the support of the basic limit functions associated to a Hermite subdivision scheme.
Let us consider a sequence of masks $(\Ab^{[n]}: n \geq 0)$ whose
support is contained in a finite interval $[-N,N]$, $N \in \NN$, i.e.,
$\supp(\Ab^{[n]}) \subseteq [-N,N]$ for all $n \in \NN$. Moreover, the
associated Hermite subdivision scheme  
$S({\Ab}^{[n]}: n \geq 0)$ is assumed to be $C^d$-convergent. Denote by $(\Fb^{[n]}: n \geq 0)$ its sequence of basic limit functions.
Using similar arguments as in \cite[Section 2.3]{dyn02}, it is easy to see that also
	\begin{equation}\label{eq:finite_support_F}
	\supp(\Fb^{[n]}) \subseteq [-N,N], \quad n \in \NN.
	\end{equation}
This fact is essential in the proof of the main theorem:

\begin{thm}\label{T:leveldep_decay}
  Let $S({\Ab}^{[n]}: n \geq 0)$ be a $C^d$-convergent interpolatory
  Hermite subdivision scheme satisfying the $V_{p,\Lambda}$-spectral
  condition.
  Moreover assume that there exists $N \in \NN$ such that
  $\supp(\Ab^{[n]}) \subseteq [-N,N]$ for all $n \in \NN$,
  and that $\sup_{n \in \NN}\|\Fb^{[n]}\|_{\infty} < \infty$.
  For $f \in C^{d}_u(\RR)$ the associated wavelet coefficients
  $\db^{[n]}$ defined in (\ref{eq:dn}) satisfy the following property:
  For $R<1$, there exist $m \in \NN$ and a constant $C>0$, depending on
  $\Lambda,R,f,N$ and the Hermite subdivision scheme, such that
  \begin{equation*}
    \|\db^{[n]}\|_{\infty}\leq C\, 2^{-n(d+1)}, \qquad n \geq m.
  \end{equation*}
\end{thm}

\begin{proof}
  Due to \eqref{eq:cardinality} and \eqref{eq:projF}, we have that
  \begin{equation}\label{eq:generalf}
    \vb_{{{\cal P}}_n f}(2^{-n}\ell)=\vb_f(2^{-n}\ell) \qquad \ell\in\ZZ, n \in \NN,
  \end{equation}
  holds true whenever $f \in C^d_u(\RR)$; note that
  Lemma~\ref{le:vp} cannot be applied here as it is only valid for
  functions in $V_{p,\Lambda}$.
  The representations \eqref{eq:wave_coeff_with_F} and
  \eqref{eq:vector_QP} allow us to express the wavelet coefficients in
  the following way 
  \begin{equation*}
    \db^{[n]}_{\ell}
    = \Db^{n+1}\,\vb_{{\cal Q}_n f}(2^{-n-1}\ell)=
    \Db^{n+1}\left(\vb_{{{\cal P}}_{n+1} f}(2^{-n-1}\ell)-\vb_{{{\cal P}}_n f}(2^{-n-1}\ell)\right),
  \end{equation*}
  and, by (\ref{eq:generalf}),
  \begin{equation*}
    \db^{[n]}_{\ell}= \Db^{n+1}\vb_{
      f}(2^{-n-1}\ell)-\Db^{n+1}\,\vb_{{{\cal P}}_n f}(2^{-n-1}\ell),
    \qquad \ell \in \ZZ.
  \end{equation*}
  Define $g\in V_{p,\Lambda}$ by $g(x):=T_{p,\Lambda}f(x - 2^{-n-1},2^{-n-1})$, $x \in \RR$.
  Then we have
  \begin{align*}
    \left|\db^{[n]}_{\ell}\right|_{\infty}
    = & \; \left|\Db^{n+1}\,\vb_{ f}(2^{-n-1}\ell)-\Db^{n+1}\,\vb_{{{\cal P}}_n f}(2^{-n-1}\ell)\right|_\infty\\
    \le &\;\left|\Db^{n+1}\vb_{ f}(2^{-n-1}\ell)-\Db^{n+1}\vb_g(2^{-n-1}\ell)\right|_\infty \\
      & +\left|\Db^{n+1}\vb_g(2^{-n-1}\ell)-\Db^{n+1}\vb_{{{\cal P}}_n f}(2^{-n-1}\ell)\right|_\infty\\
    =:&\; I_1+I_2.
  \end{align*}
  We start by estimating $I_1$. With $m \ge 1-\log_2 R$ we have
  $2^{-n-1}<R$ for $n\geq m$ and 
  Lemma~\ref{lem:errorGeneralizedTaylor} with $h=2^{-n-1}$ and $x=2^{-n-1}(\ell-1)$ gives
  \begin{align}\nonumber
    I_1 & =  \left|\Db^{n+1}\vb_{ f}(2^{-n-1}\ell)-\Db^{n+1}\vb_g(2^{-n-1}\ell)\right|_\infty \\ \nonumber
        & = \left| \left[
          2^{-(n+1)j} \left( f^{(j)}(2^{-n-1}\ell)-(T_{p,\Lambda}f)^{(j)}
          (2^{-n-1}(\ell-1),2^{-n-1}) 
          \right) \right]_{j=0}^d \right|_{\infty} \\ \label{eq:bound_I1}
        & \leq C_{\Lambda,R,f}\, 2^{-(n+1)(d+1)},  \qquad \ell \in \ZZ,
  \end{align}
  for $n \geq m$.
  The bound \eqref{eq:bound_I1} is even independent of $\ell \in \ZZ$
  since $f$ is uniformly continuous.
	
  It remains to estimate $I_2$.
  Due to Lemma~\ref{le:vp}, $\vb_g(2^{-n-1}\ell)=\vb_{{{\cal P}}_n
    g}(2^{-n-1}\ell)$ holds for $n\in \NN, \ell \in \ZZ$. By
  \eqref{eq:finite_support_F}
  there exists $N \in \NN$ such that 
  $\supp(\Fb^{[n]}) \subseteq [-N,N]$, $n \in \NN.$ Therefore,
  $\Fb^{[n]}(2^{-1}\ell-k) \neq 0$ if and only if $k \in J_{\ell} :=
  \left( \ell/2 + [-N,N] \right) \cap \ZZ$. 
  Using (\ref{eq:projF}), we get
  \begin{align} \nonumber
    I_2 &=\left|\Db \sum_{k\in\ZZ} \Fb^{[n]}(2^{-1}\ell-k)\Db^n
          \left(\vb_{g}(2^{-n}k)-\vb_{f}(2^{-n}k)\right)\right|_\infty\\ \nonumber
	&=\left|\Db \sum_{k\in J_{\ell}} \Fb^{[n]}(2^{-1}\ell-k)\Db^n
          \left(\vb_{g}(2^{-n}k)-\vb_{f}(2^{-n}k)\right)\right|_\infty\\ \nonumber
	&\leq \left|\Db\right|_{\infty} \sum_{k\in J_{\ell}} \left| \Fb^{[n]}(2^{-1}\ell-k)\right|_{\infty}
          \left|\Db^n
          \left(\vb_{g}(2^{-n}k)-\vb_{f}(2^{-n}k)\right)\right|_\infty\\ \nonumber
	&\leq \left|\Db\right|_{\infty} \left(\sum_{k\in J_{\ell}}
          \left| \Fb^{[n]}(2^{-1}\ell-k)\right|_{\infty}\right)\left(\sup_{k\in J_{\ell}} \left|\Db^n
          \left(\vb_{g}(2^{-n}k)-\vb_{f}(2^{-n}k)\right)\right|_\infty\right)
    \\ \label{eq:estimate_F}
        & \le \left|\Db\right|_{\infty} \# J_\ell
          \left\| \Fb^{[n]} \right\|_{\infty} \,
          \left(\sup_{k\in J_{\ell}} \left|\Db^n
          \left(\vb_{g}(2^{-n}k)-\vb_{f}(2^{-n}k)\right)\right|_\infty\right).
  \end{align}
  The estimate is now completely similarly to $I_1$, using the
  generalized Taylor expansion: For $R<1$ there exists $m \in \NN$
  such that for $n \geq m$ we have
  \begin{equation*}
    \sup_{k\in J_{\ell}} \left|\Db^n \left(\vb_{g}(2^{-n}k)-\vb_{f}(2^{-n}k)\right)\right|_\infty \leq C_{\Lambda,R,f} \, 2^{-n(d+1)}.
  \end{equation*}
  Therefore, by continuing from (\ref{eq:estimate_F}), we obtain:
  \begin{align}\label{eq:I2}
    I_2 \leq \left|\Db\right|_{\infty}
    (2N+1)\|\Fb^{[n]}\|_{\infty}C_{\Lambda,R,f} \, 2^{-n(d+1)}, \qquad
    n \in \NN,
  \end{align}
  and since the norm of $\Fb^{[n]}$ is bounded uniformly in $n$ we combine
  \eqref{eq:bound_I1} and \eqref{eq:I2} to obtain that for $R<1$,
  there exist $m \in \NN$ and a constant $C>0$, 
  depending on $\Lambda,R,f,N$ and the Hermite subdivision scheme, such that
  \begin{equation*}
    \left|\db^{[n]}_{\ell}\right|_{\infty}\le I_1 + I_2 \leq C\,
    2^{-n(d+1)}, \qquad \ell \in \ZZ, \, n \ge 0.
  \end{equation*}
  Since $C$ is independent of $\ell\in\ZZ$, this concludes the proof.
\end{proof}

\begin{Rem}
  The assumption $\sup_n \| \Fb^{[n]} \|_\infty < \infty$ of uniform
  boundedness of the limit functions is crucial and not easy to verify
  in general. However, like in Remark~\ref{R:AsymptHermite} it is
  valid again for asymptotically equivalent 
  schemes, see once more \cite{DynLevin95}. In particular it holds true for
  $V_{p,\Lambda}$--preserving interpolatory schemes mentioned in
  Remark~\ref{R:AsymptHermite}. Also note that 
  uniform boundedness of the supports of the $\Ab^{[n]}$ is needed for
  describing convergence like in \cite{conti17a}.
\end{Rem}

\noindent
By a careful inspection of the proof of Theorem~\ref{T:leveldep_decay}, we
can derive the following improved version of this result which shows
that the decay rate of the wavelet coefficients measures the
\emph{local} regularity of the function as it is typical for compactly
supported wavelets.

\begin{cor}\label{C:LocalDecay}
  With the assumptions as in Theorem~\ref{T:leveldep_decay}, the wavelet
  coefficients $\db^{[n]}_{\lceil 2^{n+1} x \rceil}$, $x\in \RR$,
  depend on $\vb_f (y)$, $y \in x + 2^{-n}
  \left[-N-\frac12,N+\frac12\right]$. 
\end{cor}

\begin{proof}
  Let $\ell := \left\lceil 2^{n+1} x \right\rceil$, hence $\ell \in
  \left[ 2^{n+1} x, 2^{n+1} x + 1 \right]$. The estimate of $I_1$ in
  \eqref{eq:bound_I1} depends on the values of $f$ and its derivatives
  at $2^{-n-1} \ell \in \left[ x,x+2^{-n-1} \right]$ and $2^{-n-1}
  (\ell-1) \in \left[ x-2^{-n-1},x \right]$, hence, in total on the
  behavior of $f$ on $x + 2^{-n-1} [-1,1]$. On the other hand, the set
  $J_\ell$ in the estimate of \eqref{eq:I2} of $I_2$ satisfies
  $J_\ell \subset \ell/2 + [-N,N]$, hence
  $$
  2^{-n} J_\ell \subset 2^{-n-1} \left\lceil 2^{n+1} x \right\rceil +
  2^{-n} [-N,N] \subseteq x + 2^{-n} \left[ -N,N+\frac12 \right].
  $$
  Thus, the term $\vb_f \left( 2^{-n} k \right)$, $k \in J_\ell$,
  involves only values from that interval while $\vb_g \left( 2^{-n} k
  \right)$, $k \in J_\ell$, uses $x + 2^{-n} \left[ -N-\frac12,N \right]$.
\end{proof}

\noindent
Corollary~\ref{C:LocalDecay} is the justification to use
multiwavelets for edge detection. If the sampled vector data leads to
wavelet coefficients that do not decay like $2^{-n(d+1)}$, then the
underlying function cannot be $C^d$ at a position specified by the
location of the slowly decaying wavelet coefficients. The higher the
level of the wavelet coefficients is, the better the localization of
the singularity, reproducing a well-known wavelet effect,
cf. \cite{mallat09:_wavel_tour_signal_proces}.

\section*{Conclusion}\label{sec:conclusion}
In this paper we presented results on level-dependent Hermite
subdivision schemes  preserving polynomial and exponential data,
focusing on the interpolatory case, which allows to naturally obtain
multiwavelet systems via the prediction-correction approach.  
Such wavelets possess a generalized  
vanishing moment property with respect to elements in the space
spanned by exponentials and polynomials. Vanishing moments can be crucial for data
compression purposes, in particular when such systems are applied to
data exhibiting transcendental features. 
In addition, a result concerning the decay of the wavelet coefficients
corresponding to any $f\in C^d_u(\RR)$ is proved, yielding an
analogous extension of the classical result in the standard wavelet
theory. To the best of our knowledge, this result  has never been
presented  even in the case of standard (non level-dependent) Hermite
multiwavelets. In order to prove it, a generalized Taylor formula in
the space $V_{d,\Lambda}$ is introduced, which may be of
independent interest, and error bounds on the deviation from the
classical Taylor polynomial approximation are given.
Future research includes the extension of our  results to the case of
manifold-valued data.

\section*{Acknowledgments}
This research was partially supported by the DFG Collaborative
Research Center TRR 109, ``Discretization in Geometry and Dynamics''.
Most of this research was done while the second author was with the University of Passau. The second author also thanks 
the Department of Chemical and Biological Engineering, Princeton
University, for their hospitality.
The fourth author was partially supported by the Emmy Noether
Research Group KR 4512/1-1.

\bibliographystyle{amsplain}
\bibliography{references}
\end{document}